\documentclass[12pt]{article}
\usepackage{amsmath,amsthm,amsfonts,amssymb,times}

\newtheorem{theorem}{Theorem}
\newtheorem{lemma}{Lemma}

\theoremstyle{remark}

\newtheorem*{remark*}{Remark}
\newcommand{\Q}{\mathbb{Q}}

\allowdisplaybreaks

\begin{document}

\title{On the Galois group over $\mathbb Q$ \\ of a truncated binomial expansion}

\author{Michael Filaseta\\
Department of Mathematics \\ 
University of South Carolina \\ 
Columbia, SC 29208 \\
E-mail:  filaseta@math.sc.edu
\and
Richard Moy\\
Department of Mathematics \\
Willamette University\\
Salem, OR 97301\\
E-mail:  rmoy@willamette.edu}

\date{}
\maketitle
 
\centerline{\parbox[t]{13cm}{R\'ESUM\'E: Pour tout nombre entier positif $n$, les extensions binomiales tronqu\'ees de $(1 + x)^n $ constitut\'ees de tous les termes de degr\'e $\leq r $ o\`u $1 \leq  r  \leq n-2$ semblent toujours \^etre irr\'eductibles. Pour $r$ fixe et $n$ suffisamment grand, ce r\'esultat est connu. Nous montrons ici que, pour un nombre entier positif fixe $ r \neq 6$  et $n$ suffisamment grand, le groupe Galois d'un tel polyn\^ome sur les nombres rationnels est le groupe sym\'etrique $S_r$. Pour $r = 6$, nous montrons que le nombre de $n \leq N$ exceptionnels pour lesquels le groupe Galois de ce polyn\^ome n'est pas $S_r$ est au plus $O(\log N)$.}}

\vskip 15pt \noindent
\centerline{\parbox[t]{13cm}{ABSTRACT: For positive integers $n$, the truncated binomial expansions of $(1+x)^n$ which consist of all the terms of degree $\le r$ where $1 \le r \le n-2$ appear always to be irreducible.  For fixed $r$ and $n$ sufficiently large, this is known to be the case.  We show here that for a fixed positive integer $r \ne 6$ and $n$ sufficiently large, the Galois group of such a polynomial over the rationals is the symmetric group $S_{r}$.  For $r = 6$, we show the number of exceptional $n \le N$ for which the Galois group of this polynomial is not $S_r$ is at most $O(\log N)$.}}
 
\vskip 10pt \noindent
\section{Introduction}
For $t$ and $r$ non-negative integers, we define
\[
p_{r,t}(x) = \sum_{j=0}^{r} \binom{t+j}{t} x^j = \sum_{j=0}^{r} \binom{t+j}{j} x^j.
\]
This polynomial arises from a normalization of the $t^{\rm th}$ derivative of $1+x+ \cdots +x^{t+r}$. 
The polynomial is connected to a factor of the Shabat polynomials of a family of 
\emph{dessins d'enfant} which are trees and have passport size one (cf.~\cite[Example 3.3]{NMA}). 
The polynomial $p_{r,t}(x)$ was conjectured to be irreducible, and the irreducibility was studied in \cite{BFLT}.  
In particular, we find there that if $r$ is fixed, then $p_{r,t}(x)$ is irreducible for $t$ sufficiently large.  
A generalization of this irreducibility result can be found in \cite{FKP}, where this polynomial was
considered in a different form.   There, the irreducibility of the polynomial
\[
q_{r,n}(x) = \sum_{j=0}^{r} \binom{n}{j} x^{j},
\]
which is a truncated binomial expansion of $(x+1)^{n}$, was investigated.    
As noted there, this truncated binomial expansion came up in 
investigations of the Schubert calculus in Grassmannians \cite{ISch}.   
Other results concerning these polynomials can be found in \cite{DK, DS, KKL}. 

There are some identities involving $p_{r,t}(x)$ and $q_{r,n}(x)$ which helped establish the 
results found in \cite{BFLT} and \cite{FKP}.  If we define 
\[
\tilde{p}_{r,t}(x) = x^{r}p_{r,t}(1/x) = \sum_{j=0}^{r} \binom{t+j}{j} x^{r-j},
\]
then according to \cite{BFLT} we have
\[
\tilde{p}_{r,t}(x+1) = \sum_{j=0}^{r} \binom{t+r+1}{j} x^{r-j}.
\]
Thus, $\tilde{p}_{r,t}(x+1) = x^{r} q_{r,t+r+1}(1/x)$.  We have from \cite{FKP} that
\[
q_{r,n}(x-1) = \sum_{j=0}^{r} c_{j} x^{j},
\quad \text{ where }
c_{j} = \binom{n}{j} \binom{n-j-1}{r-j} (-1)^{r-j}.
\]
As noted in \cite{FKP}, we can write
\begin{equation*}
c_{j} = \dfrac{(-1)^{r-j} n (n-1) \cdots (n-j+1) (n-j-1) \cdots (n-r+1) (n-r)}{j! (r-j)!}.
\end{equation*}

These identities are of interest as the irreducibility over $\mathbb Q$ of one of 
$p_{r,t}(x)$, $\tilde{p}_{r,t}(x)$, $\tilde{p}_{r,t}(x+1)$ and $q_{r,t+r+1}(x-1)$ implies the irreducibility of the other three.  
Furthermore, it is not difficult to see that these polynomials all have the same discriminant
(as reversing the coefficients of a polynomial and translating do not affect the discriminant).  
Also, as the roots for each all generate the same number field, we have that for
a fixed $r$ and $t$, the Galois groups
over $\mathbb Q$ associated with these polynomials are all the same.  

The main goal of this paper is to show that these polynomials give rise to examples 
of polynomials having Galois group over $\mathbb Q$ the symmetric group.  

\begin{theorem}\label{mainthm}
Let $r$ be an integer $\ge 2$ with $r \ne 6$.  If $t$ is a sufficiently large positive integer, then the Galois group
associated with any one of $p_{r,t}(x)$, $\tilde{p}_{r,t}(x)$, $\tilde{p}_{r,t}(x+1)$ and $q_{r,t+r+1}(x-1)$
over $\mathbb Q$ is the symmetric group $S_{r}$.  In the case that $r = 6$, there are at most $O(\log T)$ 
values of $t \le T$ for which the Galois group of any one of $p_{r,t}(x)$, $\tilde{p}_{r,t}(x)$, $\tilde{p}_{r,t}(x+1)$ and $q_{r,t+r+1}(x-1)$
over $\mathbb Q$ is not the symmetric group $S_{6}$.  In these cases, for sufficiently large $t$, 
the Galois group is $PGL_{2}(5)$, a transitive subgroup of $S_{6}$ isomorphic to $S_{5}$.  
\end{theorem}

\noindent
Observe that Theorem~\ref{mainthm} has as an immediate consequence that 
for a fixed integer $r \ge 2$ and $r \ne 6$, the Galois group of $q_{r,n}(x)$ over $\mathbb Q$
is $S_{r}$ provided $n$ is sufficiently large with a similar result for almost all $n$ in the case that $r = 6$. 
We note that $q_{6,10}(x)$ has Galois group $PGL_{2}(5)$.  
The proof of Theorem~\ref{mainthm} will give, up to a finite number of exceptions,
an explicit description of the set $\mathcal N$ of the $O(\log T)$ values of $t \le T$ 
where the Galois group $PGL_{2}(5)$ might occur.  
We explain computations that verify directly that for $10 < n \le 10^{10}$ and $n \in \mathcal N$, 
the Galois group of $q_{6,n}(x)$ is $S_{6}$.  The bound of $10^{10}$ can easily be extended much 
further.  However, we note that there may still be $n \in (10,10^{10}]$ for which $q_{6,n}(x)$ is reducible
so that $q_{6,n}(x)$ does not have Galois group $S_{6}$ since
the explicitly given $\mathcal N$ does not take into account that our proof
that $q_{6,n}(x)$ has Galois group $S_{6}$ requires $n$ to be sufficiently large so that, in particular, the
results from \cite{BFLT} and \cite{FKP} imply $q_{6,n}(x)$ is irreducible.  Nevertheless, based on further
computations, we conjecture that $q_{6,n}(x)$ has Galois group $S_{6}$ for all $n \ge 11$.

\section{Preliminary Material}

We will make use of Newton polygons, which we describe briefly here.  
Let $f(x) = \sum_{j=0}^{r} a_{j}x^{j} \in \mathbb Z[x]$ with $a_{0} a_{r} \ne 0$, and 
let $p$ be a prime.  
For an integer $m \ne 0$, we use $\nu_{p}(m)$ to denote the exponent in the largest power of $p$ dividing $m$.  
Let $S$ be the set of lattice points $\big(j, \nu_{p}(a_{r-j})\big)$, for $0 \le j \le r$ with $a_{r-j} \ne 0$.  
The polygonal path along the lower edges of the convex hull of these points from 
$\big(0, \nu_{p}(a_{r})\big)$ to $\big(r, \nu_{p}(a_{0})\big)$
is called the Newton polygon of $f(x)$ with respect to the prime $p$.  
The left-most edge has an endpoint $\big(0, \nu_{p}(a_{r})\big)$ and the right-most edge has an endpoint
$\big(r, \nu_{p}(a_{0})\big)$.  The endpoints of every edge belong to the set $S$, 
and each edge has a distinct slope that increases as we move along the Newton polygon from left to right.  

Newton polygons provide information about the factorization of $f(x)$ over the $p$-adic field $\mathbb Q_{p}$
and, hence, information about the Galois group of $f(x)$ over $\mathbb Q_{p}$.  As this Galois group is a subgroup
of the Galois group of $f(x)$ over $\mathbb Q$, we can use Newton polygons to obtain information about
the Galois group of $f(x)$ over $\mathbb Q$.  Recalling that we are viewing edges of Newton polygons 
as having distinct slopes, each edge of the Newton polygon of $f(x)$ with respect to
a prime $p$ corresponds to a factor of $f(x)$ in $\mathbb Q_{p}[x]$ that is not necessarily irreducible.  
More precisely, if an edge has endpoints $(x_{1},y_{1})$ and $(x_{2},y_{2})$, then its slope
$a/b = (y_{2}-y_{1})/(x_{2}-x_{1})$ with $\gcd(a,b) = 1$ is such that $f(x)$ has a factor $g(x)$ in $\mathbb Q_{p}[x]$
of degree $x_{2}-x_{1}$ and with each irreducible factor of $g(x)$ in $\mathbb Q_{p}[x]$ of degree a multiple of $b$.  

We comment here that a theorem of Dedekind (cf.~\cite{DAC}) allows one to obtain information about the 
Galois group associated with a polynomial $f(x)$ over $\mathbb Q$ by looking at the polynomials factorization 
modulo a prime $p$.  More precisely, suppose $f(x)$ is an irreducible polynomial in $\mathbb Z[x]$ and $p$ is a prime
which does not divide its discriminant.  Suppose further that $f(x)$ factors modulo $p$
as a product of $r$ irreducible polynomials of degrees $d_{1}, \ldots, d_{r}$.  Then Dedekind's Theorem asserts
that the Galois group of $f(x)$ over $\mathbb Q$ contains an element that is the product of $r$ disjoint cycles 
with cycle lengths $d_{1}, \ldots, d_{r}$.  

Our main tool for establishing Theorem~\ref{mainthm} is based on combining some of the above ideas with
a theorem of C.~Jordan \cite{CJ} and noted in work of R.~Coleman \cite{RFC}.  It has been cast in a 
convenient form by F.~Hajir \cite{FH}, which we summarize as follows.

\begin{lemma}\label{hajirlemma}
Let $f(x)$ be an irreducible polynomial of degree $r$, and suppose $q$ is a prime in the interval $(r/2,r-2)$
such that the Newton polygon with respect to some prime $p$ has an edge with slope $a/b$ where 
$a$ and $b$ are relatively prime integers and $q|b$.  Let $\Delta$ be the discriminant of $f(x)$.  
Then the Galois group of $f(x)$ over $\mathbb Q$ is the alternating group $A_{r}$ if $\Delta$ is a square 
and is the symmetric group $S_{r}$ if $\Delta$ is not a square.
\end{lemma}

We also make use of the following result from \cite{KC} and \cite{DM} (Theorem 3.3A).

\begin{lemma}\label{kconradlemma}
Let $f(x)$ be an irreducible polynomial of degree $r \ge 2$.  
If the Galois group of $f(x)$ over $\mathbb Q$ contains a $2$-cycle and a $q$-cycle for some 
prime $q > r/2$, then the Galois group is $S_{r}$.  Alternatively, if the Galois group of $f(x)$ over 
$\mathbb Q$ contains a $3$-cycle and a $q$-cycle for some 
prime $q > r/2$, then the Galois group is either the alternating group $A_{r}$ or 
the symmetric group $S_{r}$.
\end{lemma}

Note that in general 
if the Galois group of an $f(x) \in \mathbb Z[x]$ over $\mathbb Q$ is contained in an alternating group, then its 
discriminant is a square.  
Thus, in the statement of Lemma~\ref{kconradlemma}, one can conclude that the Galois group is the
symmetric group by showing that the discriminant $\Delta$ of $f(x)$ is not a square.  
In the next section, we give an explicit formula for the discriminant $\Delta$ of our polynomials in
Theorem~\ref{mainthm} and show that $\Delta$ is not a square for fixed $r$
and for $t$ sufficiently large.  In the last section, for $r \ge 8$, 
we show the existence of primes $q$ and $p$ as in Lemma~\ref{hajirlemma}.  For $r \le 7$, we appeal to 
Lemma~\ref{kconradlemma} to finish off the proof of Theorem~\ref{mainthm}.

\section{The Discriminant}
\begin{lemma}\label{discrimvaluelemma}
Let $t$ and $r$ be integers with $t \ge 0$ and $r \ge 2$.  
Let $\Delta$ be the common discriminant of $p_{r,t}(x)$, $\tilde{p}_{r,t}(x)$, $\tilde{p}_{r,t}(x+1)$ and $q_{r,t+r+1}(x-1)$.  Then
\[
\Delta = (-1)^{r(r-1)/2} \dfrac{(t+1)^{r-1} (t+r+1)^{r-1} (t+2)^{r-2} (t+3)^{r-2} \cdots (t+r)^{r-2}}{(r!)^{r-2}}.
\]
\end{lemma}

\begin{proof}
We view $t$ as a variable and work with 
\[
f_{r}(x) = q_{r,t+r+1}(x-1) = \sum_{j=0}^{r} c_{j} x^{j},
\]
where
\begin{equation}\label{discreq1}
c_{j} = \dfrac{(-1)^{r-j} (t+r+1) \cdots (t+r-j+2) (t+r-j) \cdots (t+1)}{j! (r-j)!}.
\end{equation}
To clarify, for $t$ a non-negative integer, we have
\[
c_{j} = \dfrac{(-1)^{r-j} (t+r+1)!}{j! \,(r-j)! \,t! \,(t+r-j+1)}.
\]
However, from the point of view of \eqref{discreq1}, we can view $t$ as a real variable.  

We are interested in the discriminant $\Delta$ of $f_{r}(x)$.  Observe that
\begin{equation}\label{discreq2}
\Delta = \dfrac{(-1)^{r(r-1)/2}}{c_{r}} \,\text{Res}(f_{r}, f'_{r})
= \dfrac{(-1)^{r(r-1)/2} r!}{(t+r+1)(t+r) \cdots (t+3) (t+2)} \,\text{Res}(f_{r}, f'_{r}),
\end{equation}
where $\text{Res}(f_{r}, f'_{r})$ is the resultant of $f_{r}$ and $f'_{r}$ with respect 
to the variable $x$.  We express the resultant in terms of the $(2r-1) \times (2r-1)$ Sylvester determinant
\begin{equation*} 
\text{Res}(f_{r}, f'_{r}) = 
\begin{vmatrix} 
c_{r} & c_{r-1} & c_{r-2} & \hdots & c_{1} & c_{0} & 0 & 0 & \hdots & 0 \\[3pt] 
0 & c_{r} & c_{r-1} & \hdots & c_{2} & c_{1} & c_{0} & 0 & \hdots & 0 \\[3pt] 
0 & 0 & c_{r} & \hdots & c_{3} & c_{2} & c_{1} & c_{0} &  \hdots & 0 \\[3pt] 
\vdots & \vdots & \vdots & \ddots & \vdots & \vdots & \vdots & \vdots & \ddots &\vdots \\[3pt] 
r c_{r} & (r-1) c_{r-1} & (r-2) c_{r-2} & \hdots & c_{1} & 0 & 0 & 0 & \hdots & 0 \\[3pt] 
0 & r c_{r} & (r-1) c_{r-1} & \hdots & 2 c_{2} & c_{1} & 0 & 0 & \hdots & 0 \\[3pt] 
0 & 0 & r c_{r} & \hdots & 3 c_{3} & 2 c_{2} & c_{1} & 0 & \hdots & 0 \\[3pt] 
\vdots & \vdots & \vdots & \ddots & \vdots & \vdots & \vdots & \vdots & \ddots &\vdots  
\end{vmatrix}.
\end{equation*}
Observe that there are $r-1$ rows consisting of the coefficients of $f_{r}(x)$ and 
$r$ rows consisting of the coefficients of $f'_{r}(x)$.   
For each integer $j \in [1, r]$, we see from \eqref{discreq1} that $t+r+1$ divides $c_{j}$.  
We deduce that $t+r+1$ can be factored out of each element of the first $r$ columns of 
$\text{Res}(f_{r}, f'_{r})$ to show that $(t+r+1)^{r}$ is a factor of $\text{Res}(f_{r}, f'_{r})$.  
For each $k \in \{ 1, 2, \ldots, r \}$, we also have from \eqref{discreq1} that $t+r-k+1$
divides $c_{j}$ for each integer $j \in [1,r]$ with $j \ne k$.  In particular, each of the
$r-1$ columns not containing $k c_{k}$ have each element divisible by $t+r-k+1$.  
Thus, $(t+r-k+1)^{r-1}$ divides $\text{Res}(f_{r}, f'_{r})$.  Hence,
\begin{equation}\label{discreq3}
(t+r+1)^{r} (t+1)^{r-1} (t+2)^{r-1} \cdots (t+r)^{r-1}
\end{equation}
divides $\text{Res}(f_{r}, f'_{r})$.  The product in \eqref{discreq3} as a polynomial in $t$ 
has degree $(r+1)(r-1) + 1 = r^{2}$.  Hence, from \eqref{discreq2}, we see that $\Delta$
is divisible by the polynomial
\begin{equation}\label{discreq4}
(t+r+1)^{r-1} (t+1)^{r-1} (t+2)^{r-2} (t+3)^{r-2} \cdots (t+r)^{r-2}
\end{equation}
of degree $r^{2} - r$ in $t$.  

We turn to making use of
\[
p_{r,t}(x) = \sum_{j=0}^{r} d_{j} \,x^{j},
\quad \text{ where }
d_{j} = \sum_{j=0}^{r} \dfrac{(t+j) (t+j-1) \cdots (t+1)}{j!} x^{j}.
\]
Observe that the discriminant of $f_{r}(x)$ and $p_{r,t}(x)$ are both polynomials in $t$ 
that agree at all positive integers and, hence, are identical.  
We use next that $\Delta$ is the discriminant of $p_{r,t}(x)$ to show that 
$\Delta$ cannot be divisible by a higher degree polynomial in $t$ than that
given by \eqref{discreq4}.  Taking into account the leading coefficient of $p_{r,t}(x)$, we see that
\begin{equation}\label{discreq6}
\Delta = \dfrac{(-1)^{r(r-1)/2} r!}{(t+r)(t+r-1) \cdots (t+2) (t+1)} \,\text{Res}(p_{r,t}, p'_{r,t}),
\end{equation}
where
\begin{equation*} 
\text{Res}(p_{r,t}, p'_{r,t}) = 
\begin{vmatrix} 
d_{r} & d_{r-1} & d_{r-2} & \hdots & d_{1} & d_{0} & 0 & 0 & \hdots & 0 \\[3pt] 
0 & d_{r} & d_{r-1} & \hdots & d_{2} & d_{1} & d_{0} & 0 & \hdots & 0 \\[3pt] 
0 & 0 & d_{r} & \hdots & d_{3} & d_{2} & d_{1} & d_{0} &  \hdots & 0 \\[3pt] 
\vdots & \vdots & \vdots & \ddots & \vdots & \vdots & \vdots & \vdots & \ddots &\vdots \\[3pt] 
r d_{r} & (r-1) d_{r-1} & (r-2) d_{r-2} & \hdots & d_{1} & 0 & 0 & 0 & \hdots & 0 \\[3pt] 
0 & r d_{r} & (r-1) d_{r-1} & \hdots & 2 d_{2} & d_{1} & 0 & 0 & \hdots & 0 \\[3pt] 
0 & 0 & r d_{r} & \hdots & 3 d_{3} & 2 d_{2} & d_{1} & 0 & \hdots & 0 \\[3pt] 
\vdots & \vdots & \vdots & \ddots & \vdots & \vdots & \vdots & \vdots & \ddots &\vdots  
\end{vmatrix}.
\end{equation*}
Observe that $d_{j}$ is a polynomial of degree $j$ in $t$ for each $j \in \{ 0, 1, \ldots, r \}$.  
We set $A = (a_{ij})$ to be the $(2r-1) \times (2r-1)$ matrix defining $\text{Res}(p_{r,t}, p'_{r,t})$ above, so
\[
a_{ij} = 
\begin{cases}
d_{r + i - j} &\text{if } 1 \le i \le r-1 \text{ and } i \le j \le i+r \\[3pt]
(i - j + 1) \,d_{i - j + 1} &\text{if } r \le i \le 2r-1 \text{ and } i-r+1 \le j \le i \\[3pt]
0 &\text{otherwise.}
\end{cases}
\]
We make use of the definition of a determinant to obtain
\begin{equation*} 
\text{Res}(p_{r,t}, p'_{r,t})  
= \det A 
= \sum_{\sigma \in S_{2r-1}} \bigg(  \text{sgn}(\sigma) \prod_{i = 1}^{2r-1} a_{i,\sigma(i)}  \bigg).
\end{equation*}
We show that independent of $\sigma \in S_{2r-1}$, the product $\prod_{i = 1}^{2r-1} a_{i,\sigma(i)}$ 
is a polynomial of degree at most $r^{2}$ in $t$.  
In fact, more is true.  If each $a_{i,\sigma(i)} \ne 0$, then we show that $\prod_{i = 1}^{2r-1} a_{i,\sigma(i)}$ 
is a polynomial of degree exactly $r^{2}$ in $t$.  Indeed,
for such $\sigma$, we have
\begin{gather*}
i \le \sigma(i) \le i + r \quad \text{ for } 1 \le i \le r-1, \\[3pt]
i - r + 1 \le \sigma(i) \le i \quad \text{ for } r \le i \le 2r-1,
\end{gather*}
and
\begin{align*}
\deg \bigg(  \prod_{i = 1}^{2r-1} a_{i,\sigma(i)}  \bigg) 
&= \sum_{i = 1}^{r-1} \deg( a_{i,\sigma(i)} ) + \sum_{i = r}^{2r-1} \deg( a_{i,\sigma(i)} ) \\[5pt]
&= \sum_{i = 1}^{r-1} \big(r+i - \sigma(i)\big) + \sum_{i = r}^{2r-1} \big(1+i - \sigma(i)\big) \\[5pt]
&= r^{2} + \sum_{i = 1}^{2r-1} \big(i - \sigma(i)\big) = r^{2}.
\end{align*}

We set
\[
\rho = \sum_{\sigma \in S_{2r-1}} \bigg(  \text{sgn}(\sigma) \prod_{i = 1}^{2r-1} \ell(a_{i,\sigma(i)})  \bigg),
\]
where $\ell(a_{i,\sigma(i)})$ denotes the leading coefficient of $a_{i,\sigma(i)}$.  Observe that if $\rho \ne 0$, 
then $\det A$ is a polynomial of degree $r^{2}$ with leading coefficient $\rho$.  The value of $\rho$ is the
determinant of $(2r-1) \times (2r-1)$ matrix $(\ell(a_{ij}))$.  Since
\[
\ell(a_{ij}) = 
\begin{cases}
1/(r+i-j)!  &\text{for } 1 \le i \le r-1 \text{ and } i \le j \le i+r \\[3pt]
1/(i-j)!  &\text{for } r \le i \le 2r-1 \text{ and } i-r+1 \le j \le i \\[3pt]
0 &\text{otherwise},
\end{cases}
\]
this determinant is the value of $\text{Res}(g,g')$, where 
$g(x) = \sum_{j=0}^{r} x^{j}/j!$. 
This polynomial truncation of $e^{x}$ has been studied by 
R~F.~Coleman \cite{RFC} and I.~Schur \cite{IS, IS2}.   
In particular, $g(x)$ corresponds to the generalized Laguerre polynomial $L_{r}^{(-r-1)}(x)$,
for which I.~Schur \cite{IS3} gives an explicit formula for the discriminant from which the
value of $\text{Res}(g,g')$ is easily determined (also, see \cite{RM}, Chapter~9).  From these,
we see that
\[
\text{Res}(g,g') = \dfrac{1}{(r!)^{r-1}}.
\]
We deduce that $\text{Res}(p_{r,t}, p'_{r,t})$ is a polynomial of degree $r^{2}$
in $t$ with leading coefficient $1/(r!)^{r-1}$.  From \eqref{discreq6}, we see that $\Delta$ is a polynomial 
of degree $r^{2} - r$ in $t$ which has leading coefficient $(-1)^{r(r-1)/2} /(r!)^{r-2}$.  Since \eqref{discreq4} divides $\Delta$,
the lemma follows.
\end{proof}

\begin{lemma}\label{discrimnonsquarelemma}
Let $r$ be an integer $\ge 2$.  For $t$ a non-negative integer,   
let $\Delta$ be the common discriminant of $p_{r,t}(x)$, $\tilde{p}_{r,t}(x)$, $\tilde{p}_{r,t}(x+1)$ and $q_{r,t+r+1}(x-1)$.  
Then there is a $t_{0} = t_{0}(r)$ such that for all $t \ge t_{0}$, the value of $\Delta$ is not a square.
\end{lemma}

\begin{proof}
Suppose $r \ge 2$ and $t \ge 0$ are such that $\Delta$ is a square.  
From Lemma~\ref{discrimvaluelemma}, we see that $r \ge 4$ since $\Delta < 0$ for $r \in \{ 2, 3 \}$.  
We consider even and odd $r$ separately and only the case that $\Delta \ge 0$ since $\Delta < 0$ 
cannot be a square.

In the case that $r$ is even, Lemma~\ref{discrimvaluelemma} implies that
$(t+1)(t+r+1)$ is an integer that is a rational square.  
Hence, $(t+1)(t+r+1)$ is the square of an integer.  Let $\delta = \gcd(t+1,t+r+1)$.  
Then $\delta$ divides the difference $(t+r+1) - (t+1) = r$, so $\delta \le r$.  
Also $(t+1)/\delta$ and $(t+r+1)/\delta$ are relatively prime numbers whose 
product is a square, so each of them is a square. 
As $(t+r+1)/\delta - (t+1)/\delta = r/\delta \le r$ and the difference of 
two consecutive squares $(n+1)^{2}-n^{2} = 2n+1$ tends to infinity with $n$, 
we deduce that $(t+1)/\delta$ is bounded.  In fact, taking $n^{2} = (t+1)/\delta$, 
we see that
\begin{align*}
2 \sqrt{\dfrac{t+1}{r}} + 1 \le 2 \sqrt{\dfrac{t+1}{\delta}} + 1 \le \dfrac{r}{\delta} \le r
&\implies 
\dfrac{4(t+1)}{r} \le (r-1)^{2} \\[5pt]
&\implies 
t < t+1 \le \dfrac{r \,(r-1)^{2}}{4}.
\end{align*}
Thus, for $r$ even and $t \ge r (r-1)^{2}/4$, we have that $\Delta$ is not a square.

In the case that $r$ is odd, 
Lemma~\ref{discrimvaluelemma} implies that 
the largest factor of the product 
\[
(t+2)(t+3) \cdots (t+r)
\]
relatively prime to $r!$ is a square.   
As $(t+r) - (t+2) = r-2$, we see also that for every prime $p > r$
dividing the product $(t+2)(t+3) \cdots (t+r)$, there is a unique $j \in \{ 2, 3, \ldots, r \}$
for which $p|(t+j)$.  Thus, for such a $p$ and $j$, there is a positive integer $e$
for which $p^{2e} \Vert (t+j)$.  As $r \ge 5$, we deduce that there are positive integers 
$a$, $b$ and $c$ each dividing the product of the primes up to $r$ and satisfying
\[
t+2 = a u^{2}, \quad
t+3 = b v^{2} \quad \text{and} \quad
t+4 = c w^{2}, 
\]
for some positive integers $u$, $v$ and $w$.  We deduce that
\[
b^{2} v^{4} - 1 = (t+3)^{2} -1 = (t+2)(t+4) = ac (uw)^{2}.
\]
As $a$, $b$ and $c$ divide the product of the primes up to $r$, 
there are finitely many equations of the form $ac y^{2} = b^{2} x^{4} - 1$ possible
for a given $r$.  Each such equation is an elliptic curve containing finitely many 
integral points by a theorem of Siegel \cite{LM, ST, LW}.  Hence, for a fixed $r$, 
the value of $x = v = \sqrt{(t+3)/b}$ is bounded from above for every possible $b$.  
Thus, $t_{0}$ exists in
the case of $r$ odd, completing the proof.
\end{proof}


\section{Proof of Theorem~\ref{mainthm}}
As noted in the introduction, from \cite{BFLT} and \cite{FKP}, 
for $t$ sufficiently large, the polynomial
$p_{r,t}(x)$ and, hence, the polynomials $\tilde{p}_{r,t}(x)$, $\tilde{p}_{r,t}(x+1)$ and $q_{r,t+r+1}(x-1)$ are irreducible.  

We begin by considering the case that $r$ is a fixed integer $\ge 8$.  
From Lemma~\ref{hajirlemma} and Lemma~\ref{discrimnonsquarelemma}, 
it suffices to show that there is a prime $q$ in the interval $(r/2,r-2)$
such that the Newton polygon of one of 
$p_{r,t}(x)$, $\tilde{p}_{r,t}(x)$, $\tilde{p}_{r,t}(x+1)$ and $q_{r,t+r+1}(x-1)$ 
with respect to some prime $p$ has an edge with slope $a/b$ where 
$a$ and $b$ are relatively prime integers and $q|b$.  

Since $r \ge 8$, one can show using explicit results on the distribution of primes (cf.~\cite{RS})
that there is a prime $q \in (r/2,r-2)$.  Alternatively, from \cite{SR}, one has that there are
$3$ primes in the interval $(r/2,r]$ for $r \ge 17$ so that there must be at least $1$ prime
in the interval $(r/2,r-2)$ for $r \ge 17$.  Then a simple check leads to such a prime
for $r \ge 8$.

With $q$ a prime in $(r/2,r-2)$, we consider $t \in \mathbb Z^{+}$ sufficiently large. 
Note that the numbers $t+r+1-q$ and $t+1+q$ are distinct positive integers.  
Let $p$ be a prime $> r$, and suppose $p^{e} \Vert (t+r+1-q)(t+1+q)$ where $e \in \mathbb Z^{+}$.  
Observe that $p > r$ implies either $p^{e} \Vert (t+r+1-q)$ or $p^{e} \Vert (t+1+q)$.  
We use that in fact $p$ can divide at most one of $t+r+1, t+r, \ldots, t+1$.  
Suppose $p^{e} \Vert (t+r+1-q)$.  
Then the Newton polygon of $q_{r,t+r+1}(x-1)$ with respect to $p$  
consists of two edges, one joining $(0,e)$ to $(r-q,0)$ and
the other joining $(r-q,0)$ to $(r,e)$.  
In the case that $p^{e} \Vert (t+1+q)$, the Newton polygon of $q_{r,t+r+1}(x-1)$ with respect to $p$  
also consists of two edges, one joining $(0,e)$ to $(q,0)$ and
the other joining $(q,0)$ to $(r,e)$.
In either case, we see that the Newton polygon of $q_{r,t+r+1}(x-1)$ with respect to $p$  
has an edge of slope $\pm e/q$.  From Lemma~\ref{hajirlemma}, we can therefore
deduce for sufficiently large $t$, the Galois group of $q_{r,t+r+1}(x-1)$ is $S_{r}$
unless $q \mid e$.  This is true for each prime $p > r$ with $p|(t+r+1-q)(t+1+q)$.  

So suppose then that for every prime $p > r$ dividing $(t+r+1-q)(t+1+q)$, we have 
$p^{e} \Vert (t+r+1-q)$ or $p^{e} \Vert (t+1+q)$ for some $e$ divisible by $q$.   
We deduce that we can write
\[
t+1+q = a u^{q}
\quad \text{and} \quad
t+r+1-q = b v^{q},
\]
where $a$, $b$, $u$ and $v$ are positive integers with both $a$ and $b$ dividing
\[
\mathcal P = \prod_{\substack{p \le r \\ p \text{ prime}}} p^{q-1}.
\]
Note that $q \in (r/2,r-2)$ and $r \ge 8$, so that $q \ge 5$.  
For fixed $a$ and $b$ dividing $\mathcal P$, we have $u$ and $v$ must be solutions
to the Diophantine equation
\[
a u^{q} - b v^{q} = 2q - r > 0.
\]
As this is a Thue equation, we deduce that there are finitely many integral solutions in
$u$ and $v$ (cf.~\cite{ShT}).  This is true for each fixed $a$ and $b$ dividing $\mathcal P$.  
As $\mathcal P$ and $q$ only depend on $r$ and $r$ is fixed, we deduce that there are finitely
many possibilities for $t+1+q = a u^{q}$.  Hence, for sufficiently large $t$, we deduce that
$q \nmid e$ for some prime $p > r$ with $p^{e} \Vert (t+r+1-q)$ or $p^{e} \Vert (t+1+q)$.  
Consequently, in the case that $r \ge 8$, we can conclude the Galois group of 
$q_{r,t+r+1}(x-1)$ is $S_{r}$, from which the same follows for the polynomials 
$p_{r,t}(x)$, $\tilde{p}_{r,t}(x)$ and $\tilde{p}_{r,t}(x+1)$. 

Now, we consider the case that $r \le 7$.  
We consider $t$ sufficiently large so that in particular the polynomials in Theorem~\ref{mainthm} are irreducible.  
In the case that $r = 2$ , the only possibility then is that the Galois group is $S_{2}$.  
For $r=3$, we use also that, by Lemma~\ref{discrimnonsquarelemma}, 
the discriminant of $p_{3,t}(x)$ is not a square, and this is enough to imply that 
the Galois group of $p_{3,t}(x)$ over $\mathbb Q$ is $S_{3}$.  
For $r = 4$, suppose $p$ is a prime $> 3$ dividing $(t+2)(t+4)$.  If $p^{e}\Vert (t+2)(t+4)$, then
$p^{e}\Vert (t+2)$ or $p^{e} \Vert (t+4)$.  In either case, we see that the Newton polygon of 
$q_{r,t+r+1}(x-1)$ with respect to $p$ consists of an edge with slope $e/3$.  
If $3 \nmid e$, then the Galois group will have a $3$-cycle so that Lemma~\ref{kconradlemma}
and Lemma~\ref{discrimnonsquarelemma} imply that the Galois group is $S_{4}$.  
Otherwise, $3 \mid e$ for each prime $p > 3$ dividing
$(t+2)(t+4)$.  We deduce that $t+4 = a u^{3}$ and $t+2 = b v^{3}$ where $a$ and $b$ divide $36$.  
Observe that $a u^{3} - b v^{3} = 2$.  
This is a Thue equation, and as before this equation has no solutions for $t$ sufficiently large.  Thus,
since $t$ is sufficiently large, the Galois group of $p_{4,t}(x)$ over $\mathbb Q$ is $S_{4}$.

For $r = 5$ and $r = 7$, one can give similar arguments.  Specifically, for $r = 5$ and 
for a prime $p > 3$ such that $p^{e}\Vert (t+3)(t+4)$, we deduce that
either $3 \mid e$ or else there is a $\sigma$ in the Galois group of $p_{5,t}(x)$ over $\mathbb Q$
which is a $3$-cycle or is a product of two disjoint cycles, one a $3$-cycle and one a $2$-cycle.  
In the case that $3 \mid e$ for every such prime $p > 3$, 
we have $t+4 = a u^{3}$, $t+3 = b v^{3}$ and $a u^{3} - b v^{3} = 1$, 
where $a$ and $b$ divide $36$.  Since $t$ is sufficiently large, this does not occur.  
In the case that $\sigma$ is a product of a $3$-cycle and a $2$-cycle, we see that
$\sigma^{2}$ is a $3$-cycle.  Thus, regardless of $\sigma$, we can apply Lemma~\ref{kconradlemma}
and Lemma~\ref{discrimnonsquarelemma} to deduce that the Galois group of $p_{5,t}(x)$ over $\mathbb Q$ is $S_{5}$.  
For $r = 7$ and for a prime $p > 3$ such that $p^{e}\Vert (t+4)(t+5)$, one similarly argues
that either $3 \mid e$ for every prime $p > 3$ and a Thue equation shows that this impossible
since $t$ is sufficiently large or there is a $\sigma$ in the Galois group of $p_{7,t}(x)$ over $\mathbb Q$
such that $\sigma^{4}$ is a $3$-cycle.  Also, for $r = 7$ and for a prime 
$p > 7$ such that $p^{e}\Vert (t+1)(t+8)$, one similarly argues
that either $7 \mid e$ for every prime $p > 7$ and a Thue equation shows that this impossible
since $t$ is sufficiently large or there is a $7$-cycle in the Galois group of $p_{7,t}(x)$ over $\mathbb Q$. 
Thus, the Galois group of $p_{7,t}(x)$ over $\mathbb Q$ contains a $3$-cycle and a $7$-cycle, 
and Lemma~\ref{kconradlemma} and Lemma~\ref{discrimnonsquarelemma} imply this Galois group is $S_{7}$.

We are left with the case that $r = 6$.  
There are $16$ transitive subgroups of $S_{6}$ (cf.~\cite{DM}).  
We can eliminate all but two of these as possibilities for 
the Galois group $G$ of $p_{6,t}(x)$ over $\mathbb Q$ as follows.
Using an argument similar to the above, we consider a prime $p > 5$
such that $p^{e}\Vert (t+2)(t+6)$ to show that 
either $5 \mid e$ for every prime $p > 5$ and a Thue equation shows that this impossible
since $t$ is sufficiently large or there is a $5$-cycle in $G$. 
Since $t$ is sufficiently large, we deduce that $p_{6,t}(x)$ is irreducible over $\mathbb Q$, 
$p_{6,t}(x)$ has a non-square discriminant in $\mathbb Q$, and $G$ contains a $5$-cycle. 
The latter implies that $5$ divides $|G|$.   
Of the $16$ transitive subgroups of $S_{6}$, only $4$ have size divisible by $5$, and of those
exactly $2$ are contained in $A_{5}$.  Since the discriminant of $p_{6,t}(x)$ is not a square,
this leaves then just $2$ possibilities for $G$, 
one is $S_{6}$ and the other is $PGL_{2}(5)$, which is a subgroup of $S_{6}$ that is 
isomorphic to $S_{5}$.  

For the purposes of the proof of Theorem~\ref{mainthm}, we can distinguish between 
cases where $G = S_{6}$ and cases where $G = PGL_{2}(5)$ by observing that 
$S_{6}$ has an element which is the product of two disjoint cycles, one a $2$-cycle 
and the other a $4$-cycle, whereas $PGL_{2}(5)$ has no such element.
We consider a prime $p > 3$ such that $p^{e}\Vert (t+3)(t+5)$ for some $e \in \mathbb Z^{+}$.    

If $2 \nmid e$ then ${p}_{6,t}(x)=g(x)h(x)$ where $g(x)$ and $h(x)$ are irreducible polynomials over $\Q_p$ of degrees $2$ and $4$ respectively. Let $F_g$ and $F_h$ denote the splitting fields of $g$ and $h$ over $\Q_p$ and observe that they are tamely ramified since $p>3$. Using Newton polygons, one deduces that $F_g$ is totally ramified and that the ramification index of $F_h$ is divisible by $4$. We know that $F_h$ is tamely ramified and therefore the tame inertia group is cyclic with order divisible by $4$ \cite[Corollary 1, p. 31]{CF}. Since $S_4$ has no larger cyclic subgroups, we deduce that the ramification index of $F_h$ is exactly $4$ and the tame inertia subgroup is generated by a $4$-cycle (the only possible form of an element with order $4$ in $S_4$).

Now, let $K$ be the compositum of $F_g$ and $F_h$. If $F_h\subsetneq K$, then $F_h\cap F_g = \Q_p$. Therefore, there is an element of the Galois group of ${p}_{6,t}(x)$ that permutes the 2 roots of $g$ and cyclicly permutes the 4 roots of $h$. That is, the Galois group of ${p}_{6,t}(x)$ contains an element which is the disjoint product of a 4-cycle and a 2-cycle.

If $K=F_h$, then $F_g\subset F_h$. Let $\tau$ be a generator of the tame inertia group of $F_h$. If $\tau$ permutes the roots $g$, then the Galois group of ${p}_{6,t}(x)$ contains an element which is the disjoint product of a 4-cycle and a 2-cycle. If $\tau$ fixes the roots of $g$, then the roots of $g$ lie in $K^\tau$, the fixed field of $\tau$. However, since $\tau$ generates the inertia subgroup of $K$, we know that $K^\tau$ is an unramified extension of $\Q_p$ \cite[Proposition 9.11, p. 173]{JN}. Therefore, $F_g\subset K^\tau$ is unramified, which is a contradiction with our previous deduction that $F_g$ is a totally ramified quadratic extension of $\Q_p$. 

Since the Galois group $G$ has an element that is the product of two disjoint cycles, one a $2$-cycle 
and the other a $4$-cycle, we have shown that $G = S_{6}$.  In the case that $2 \mid e$
for every prime $p > 3$, we have $t+5 = a u^{2}$ and $t+3 = b v^{2}$ where $a$ and
$b$ are divisors of $6$.  In this case we have, for fixed $a$ and $b$, the Diophantine equation
\begin{equation}\label{pelleq}
a u^{2} - b v^{2} = 2
\end{equation}
in the variables $u$ and $v$.  
\begin{table}[h]
\centering
\renewcommand{\tabcolsep}{8pt}
\renewcommand{\arraystretch}{1.2}
\begin{tabular}{|c|c|}
\hline
Pairs $(a,b)$ & All Solutions \\ \hline \hline
$(1,2)$ & $u = 2u'$ where $(1+\sqrt{2}\,)^{2m-1} = v + \sqrt{2}\,u'$ for $m \in \mathbb Z^{+}$ \\ \hline
$(2,1)$ & $v = 2v'$ where $(1+\sqrt{2}\,)^{2m} = u + \sqrt{2}\,v'$ for $m \in \mathbb Z^{+}$ \\ \hline
$(2,3)$ & $v = 2v'$ where $(5+2\sqrt{6}\,)^{m} = u + \sqrt{6}\,v'$ for $m \in \mathbb Z^{+}$ \\ \hline
$(2,6)$ & $(2+\sqrt{3}\,)^{m} = u + \sqrt{3}\,v$ for $m \in \mathbb Z^{+}$ \\ \hline
$(3,1)$ & $(1+\sqrt{3}\,)(2+\sqrt{3}\,)^{m-1} = v + \sqrt{3}\,u$ for $m \in \mathbb Z^{+}$ \\ \hline
$(6,1)$ & $(2+\sqrt{6}\,)(5+2\sqrt{6}\,)^{m-1} = v + \sqrt{6}\,u$ for $m \in \mathbb Z^{+}$ \\ \hline
\end{tabular}
\caption{Solutions to the Pell Equations}\label{table}
\end{table}
Of the $16$ possibilities for $(a,b)$ where $a$ and $b$ divide $6$,
there are $9$ for which \eqref{pelleq} can be shown to have no solutions modulo either $3$ or $4$.  
For $(a,b) = (2,2)$, the equation \eqref{pelleq} is equivalent to $u^{2}-v^{2} = 1$.  Since consecutive
positive squares differ by more than $1$ and since $t+5 = a u^{2}$, we deduce that there are no solutions 
for $t \ge 1$.  The remaining $6$ choices of $(a,b)$ are tabulated in Table~\ref{table}.  
Here, the equation \eqref{pelleq} corresponds to a 
Pell equation which has infinitely many solutions in \textit{positive} integers $u$ and $v$ 
given by the right column in the table.
These solutions were found using classical methods for solving Pell equations (cf.~\cite{GC}), and
we do not elaborate on the details.  In each case, the solutions grow exponentially, and the total number
of solutions in pairs $(u,v)$ with $u$ and $v$ each $\le X$ is $O(\log X)$.  As 
$t+5 = a u^{2}$ and $t+3 = b v^{2}$, we deduce that the number of $t \le T$ such that $G = PGL_{2}(5)$
is at most $O(\log T)$, completing the proof of Theorem~\ref{mainthm}.

The inclusion of the phrase ``at most" in the theorem is to emphasize that we do not know that 
these exceptional pairs that arose at the end of this proof give rise to cases where $G = PGL_{2}(5)$.  
In fact, it is likely that $G = S_{6}$ for every sufficiently large $t$ when $r = 6$.  For $t \in \{ 1,3 \}$,
which arise from the two smallest solutions coming from Table~\ref{table}, one checks that 
$G = PGL_{2}(5)$.  There are $37$ other positive integer values of $t \le 10^{10}$ coming from Table~\ref{table},
and one checks that for each of these we have:

\begin{itemize}
\item
For some prime $p_{1} \le 149$, the polynomial $p_{6,t}(x)$ is an irreducible sextic polynomial modulo $p_{1}$.  
Hence, $p_{6,t}(x)$ is irreducible.
\item
The discriminant $\Delta$ of $p_{6,t}(x)$ is not a square.  Hence, 
$G$ is not contained in $A_{6}$.  (Note that by Lemma~\ref{discrimvaluelemma}, if $r = 6$, then $\Delta < 0$ 
for all non-negative integers $t$; thus, $\Delta$ cannot be a square if $r = 6$.)

\item
For some prime $p_{2} \le 101$, the polynomial $p_{6,t}(x)$ factors as a linear polynomial times
an irreducible quintic modulo $p_{2}$.  Hence, $G = PGL_{2}(5)$ or $G = S_{6}$. 
\item
For some prime $p_{3} \le 109$ not dividing the discriminant $\Delta$ of $p_{6,t}(x)$, 
the polynomial $p_{6,t}(x)$ factors as an irreducible quadratic times
an irreducible quartic modulo $p_{3}$.  Hence, $G = S_{6}$ (using Dedekind's Theorem discussed in Section 2). 
\end{itemize}

\noindent
It therefore is plausible that for $t > 3$ in general, the Galois group of $p_{6,t}(x)$ is in fact $S_{6}$.
Note that since $\tilde{p}_{r,t}(x+1) = x^{r} q_{r,t+r+1}(1/x)$, the comments about $q_{r,n}(x)$ 
after the statement of Theorem~\ref{mainthm} follow.

\end{document}